\documentclass[12pt]{article}

\usepackage{amsfonts,amsmath,amsthm,bm,bbm, thmtools, thm-restate}
\usepackage{latexsym,color,epsfig,mathrsfs,enumerate}
\usepackage[toc,page]{appendix}
\usepackage{fancyhdr}
\usepackage[colorlinks=true, allcolors=red]{hyperref}
\usepackage{indentfirst}
\usepackage[maxbibnames=9]{biblatex}
\usepackage{comment}
\usepackage{xcolor}
\usepackage[shortlabels]{enumitem}

\usepackage{graphicx}
\usepackage{pgfplots}
\pgfplotsset{compat=1.15}
\usetikzlibrary{arrows}

\nocite{*}

\newtheorem{thm}{Theorem}[section]
\newtheorem{dfn}[thm]{Definition}
\newtheorem{lem}[thm]{Lemma}
\newtheorem{cor}[thm]{Corollary}
\newtheorem{prop}[thm]{Proposition}

\newcommand{\dd}[2]{\operatorname{deg-diff}^{#2}({#1})}
\newcommand{\homm}[1]{\operatorname{hom}(#1)}

\usepackage{combelow}

\usepackage{newunicodechar}
\newunicodechar{ș}{\cb{s}}
\newunicodechar{ț}{\cb{t}}

\newcommand{\ov}{\overline}

\newcommand{\bP}{\mathbb P}
\newcommand{\bE}{\mathbb E}

\title{A bipartite version of the Erd\H{o}s--McKay conjecture}
\date{}
\author{Eoin Long\thanks{School of Mathematics, University of Birmingham, UK. Email: \texttt{e.long@bham.ac.uk}.} \and Lauren\cb{t}iu Ploscaru\thanks{School of Mathematics, University of Birmingham, UK. Email: \texttt{ixp090@student.bham.ac.uk}.\newline \hspace*{1.5em} 
The second author is grateful for support through an EPSRC DTP studentship.}}

\begin{document}

\maketitle

\begin{abstract}
	An old conjecture of Erd\H{o}s and McKay states that 
	if all homogeneous sets in an $n$-vertex graph are of 
	order $O(\log n)$ then the graph contains 
	induced subgraphs of each 
	size from $\{0,1,\ldots, \Omega (n^2)\}$. We 
	prove a bipartite analogue of the conjecture: 
	if all {balanced} homogeneous sets in an 
	$n \times n$ bipartite graph are of order
	$O(\log n)$ then the graph contains 
	induced subgraphs of each size from 
	$\{0,1,\ldots, \Omega (n^2)\}$.
\end{abstract}

\section{Introduction}

Given a graph $G$ we write $\mbox{hom}(G)$ to denote the homogeneous number of $G$, given by: 
    \begin{align*}
        \homm{G} 
            := 
        \max \big \{t \in {\mathbb N}: \exists\ U\subset V(G) \mbox{ with } |U| = t \mbox{ such that } G[U] \mbox{ is complete or empty} \big \}.
    \end{align*}
\indent In its simplest form, Ramsey's theorem \cite{ramsey1930}, \cite{Erdosold} states that any $n$-vertex graph $G$ satisfies $\homm {G} = \Omega(\log n)$ and a classical result of Erd\H{o}s \cite{Erdosprob} shows that this behaviour is essentially optimal; there are $n$-vertex \emph{Ramsey graphs} $G_0$ with $\homm{G_0} = O (\log n)$. However, the existence of such Ramsey graphs has only been demonstrated indirectly via probabilistic methods and finding explicit constructions of graphs exhibiting such behaviour remains a tantalising open problem (a \$1000 Erd\H{o}s problem \cite{chung1998erdos}).

Despite the challenges in constructing Ramsey graphs, and perhaps influenced by them, there has been much success in understanding the intrinsic properties possessed by these graphs. For example, Ramsey graphs have been shown to (roughly) exhibit similar behaviour to the Erd\H{o}s-Renyi random graph w.h.p. with respect to edge density \cite{erdosmrd}, non-isomorphic induced subgraphs \cite{Shelah}, universality of small induced subgraphs \cite{PROMEL}, and the possible edge sizes and degrees appearing in induced subgraphs \cite{kwan2018ramsey}, \cite{kwan2019proof}, \cite{JKLY}. 

A challenging remaining problem in this context is the Erd\H{o}s--McKay conjecture \cite{erdos1992some}. Informally, the conjecture asks whether every Ramsey graph must contain (essentially) the entire interval of possible induced subgraph sizes. More precisely the conjecture asks whether every $n$-vertex graph $G$ with $\homm{G} \leq C \log n$ satisfies $\{0,\ldots, \Omega _C(n^2)\} \subset \{e(G[U]): U \subset V(G)\}$. The best known bound for the conjecture is due to Alon, Krivelevich and Sudakov \cite{AKS-size} who proved that such graphs necessarily contain induced subgraphs of each size in $\{0, \ldots, n^{\Omega _C(1)}\}$. The conjecture was also proved for random graphs (in a strong form) by Calkin, Frieze and McKay in \cite{calkin1992subgraph}. More recently, Kwan and Sudakov \cite{kwan2018ramsey} gave further support proving that such graphs necessarily contain induced subgraphs of $\Omega _C(n^2)$ different sizes, which improved an earlier almost quadratic bound of Narayanan, Sahasrabuhde and Tomon \cite{narayanan2019ramsey}.

Our aim here is to study the natural analogue of this conjecture for bipartite graphs. Recall that the bipartite analogue of Ramsey's theorem states that any balanced bipartite graph $G = (V_1, V_2, E)$ with $|V_1| = |V_2| = n$ contains either $K_{t,t}$ or $\overline{K_{t,t}}$ as an induced subgraph with $t = \Omega (\log n)$. This type of behaviour is again known to be the best possible in general, though explicit constructions of `bipartite Ramsey graphs' are also unknown. In fact, these are more challenging in a sense as such constructions would lead to constructions in the usual Ramsey setting (see e.g. \cite{barak20122}). In part, this has contributed to significant interest in Ramsey results in the bipartite setting, e.g. see \cite{conlon2008new}, \cite{collins2016zarankiewicz}, \cite{axenovich2021bipartite} \cite{axenovich2021large}, \cite{bucic20193}, \cite{souza2021blowup}. 

Given the context above, and the difficulties in the Erd\H{o}s--McKay conjecture, it is natural to ask what can be said about the edge sizes of induced subgraphs in balanced $n$-vertex bipartite Ramsey graphs? A general result of Narayanan, Sahasrabuhde and Tomon \cite{narayanan2017multiplication} gives some information here. These authors also studied a generalisation of the Erd\H{o}s `multiplication table problem', and showed that any bipartite graph with $m$ edges has induced subgraphs of $\widetilde{\Omega }(m)$ distinct sizes. Recently Baksys and Chen \cite{baksys2021number} raised the bipartite Ramsey question and proved an analogue of Kwan and Sudakov's theorem: any balanced bipartite Ramsey graph on vertex classes of order $n$ has induced subgraphs of $\Omega (n^2)$ different edge sizes. 

Our main theorem extends this line of research, proving an analogue of the Erd\H{o}s--McKay conjecture in the bipartite setting. Before stating it, we give a more precise definition of the Ramsey property for bipartite graphs, in a slightly more general setting.

\begin{dfn}
Given $C>0$, a graph $G = (V_1, V_2, E)$ is called $C$\emph{-bipartite-Ramsey} if for any $t_1\geq C\log_2|V_1|$ and $t_2\geq C\log_2|V_2|$ there is no induced copy of $K_{t_1,t_2}$ or $\ov{K_{t_1,t_2}}$ in $G$.
\end{dfn}

\indent We will often simply say that $G$ is a $C$-Ramsey graph when it is clear that $G$ is bipartite. 

We can now state our main result, which gives an analogue of the Erd\H{o}s--McKay conjecture in the bipartite Ramsey setting.

\begin{thm}\label{main-theorem}
Given $C>0$ there is $\alpha >0$ such that the following holds. Suppose that $G = (V_1, V_2,E)$ is a $C$-bipartite-Ramsey graph. Then $\{0,\ldots, \alpha |V_1||V_2|\} \subset \{e(G[U]): U \subset V(G)\}$.
\end{thm}

Our proof of Theorem \ref{main-theorem} follows a similar line of approach to \cite{narayanan2019ramsey}, \cite{kwan2018ramsey} and \cite{baksys2021number}, in which we first show that one can get close to the desired edge sizes and then refine this to show that certain perturbations are typically available to allow one to adjust to the exact size. 
Anti-concentration estimates are a key tool in ensuring that the desired perturbations are `sufficiently rich' here. We prove such bounds using diversity of  vertices and diversity for pairs of vertices, introduced by Bukh and Sudakov \cite{bukh} and Kwan and Sudakov \cite{kwan2018ramsey} respectively, though a different notion of pair diversity was key here in obtaining the required behaviour. We were also able to keep track of the perturbation structure using certain sumset techniques.\vspace{1mm}

\noindent \textbf{Update:} After this paper was submitted, Kwan, Sah, Sauermann and Sawhney \cite{KSSS} uploaded a remarkable paper, which has completely resolved the Erd\H{o}s $-$ McKay conjecture. The proof is a tour de force, combining a wide range of techniques, and it is significantly different from our approach here.\vspace{1mm}

\noindent \textbf{Notation.} Given disjoint sets $V_1$ and $V_2$ we write $G = (V_1, V_2 ,E)$ to represent a bipartite graph $G$ with vertex set $V(G)=V_1\sqcup V_2$ and edge set $E(G)=E\subset V_1\times V_2$. The \emph{edge density} of $G$ is given by $e(G)\large{/}|V_1||V_2|$. Given $U_i \subset V_i$ for $i = 1,2$ we write $G[U_1, U_2]$ to denote the induced subgraph $G[U_1, U_2] = (U_1, U_2, E \cap (U_1 \times U_2))$.

Given a graph $G$ and $u,v\in V(G)$, we write $u\sim v$ if $u$ and $v$ are adjacent vertices in $G$. The neighbourhood of $u$ is given by $N_G(u) = \{v\in V(G): u\sim v\}$ and given $S \subset V(G)$ we let $N^S_G(u) := N_G(u) \cap S$; we will omit the subscript $G$ when the graph is clear from the context. We write $d^S_G(u) = |N^S_G(u)|$.

\indent Given vertices $u,v\in V(G)$ we write $\text{div}_G(u,v)$ for the symmetric difference $N(u)\triangle N(v)$. The \emph{biased diversity} of $u$ and $v$, denoted by $\text{divb}_G(u,v)$, is going to be the largest of the two sets $N(u)\setminus N(v)$ and $N(v)\setminus N(u)$. If these have the same size then we arbitrarily pick one of the sets to be $\text{divb}_G(u,v)$. Clearly $|\text{divb}_G(u,v)|\geq |\text{div}_G(u,v)|/2$.

\indent We will also be interested in \emph{ordered pairs} of vertices $\bm{p}= (u,v) \in \binom {V(G)}{2}$. Naturally, given such a pair $\bm{p}$ we can define $\text{div}_G(\bm{p}):=\text{div}_G(u,v)$. We will also make the convention throughout that all pairs are ordered so that $\text{divb}_G(\bm{p})=N(u) \setminus N(v)$, which implies $d(u) \geq d(v)$. Moreover, for later diversity purposes, we will need to study $\dd {\bf p}{S}:=d_G^S(u)-d_G^S(v)$.

Given integers $m \leq n$ we will write $[m,n]$ to denote the interval $\{m, m+1, \ldots, n\}$ and given $n \in {\mathbb N}$ we write $[n]$ for the interval $\{1,\ldots, n\}$. All logarithms in the paper will be base $2$ unless otherwise stated. Floor and ceiling signs are omitted throughout for clarity of presentation.

\section{Collected tools}

Before beginning in earnest on the proof of the theorem, we make two simple observations on the relation between vertex classes and the size of a $C$-bipartite-Ramsey graph $G = (V_1, V_2, E)$, which are useful for future reference. 
    \begin{itemize}
        \item Given any $\alpha \in (0,1)$ and $U_i \subset V_i$ with $|U_i| \geq |V_i|^{\alpha }$ for $i = 1,2$, the induced graph $G[U_1, U_2]$ is $(C \alpha ^{-1})$-Ramsey.
        \item Suppose that $|V_1| \geq |V_2|$ and let $W \subset V_2$ of size
        $|W| = 2C \log |V_2|$. By the pigeonhole principle there is a set $U \subset V_1$ with $|U| \geq |V_1|2^{-|W|}$ 
        such that for all $u\in U$ one has $N_G^W(u) = W'$. 
        It follows that $G$ contains an induced $K_{t_1, t_2}$ or $\overline {K_{t_1,t_2}}$ where $t_1 = |V_1|2^{-|W|}$ 
        and $t_2 = C \log |V_2|$. As $G$ is $C$-Ramsey it follows that 
        $t_1 \leq C \log |V_1|$, which in particular gives $|V_1| \leq |V_2|^{O_C(1)}$. Thus the vertex classes of $C$-bipartite-Ramsey graphs are necessarily polynomially related. 
    \end{itemize}

\subsection{Density control}

We start by proving a density result for bipartite Ramsey graphs, which is the analog of the Erd\H{o}s-Szemer\'{e}di theorem \cite{erdosmrd} for Ramsey graphs in general. The proof uses a well-known argument due to Kovari$-$S\'os$-$Tur\'an \cite{Kovari1954}. 

\begin{lem}\label{density-lemma}
Given $C>1$ there is $n_C\in \mathbb{N}$ such that the following holds. Suppose that $G = (V_1, V_2,E)$ is a $C$-bipartite-Ramsey graph with $|V_1|,|V_2|\geq n_C$. Then $G$ has edge density between $(16C)^{-1}$ and $1-(16C)^{-1}$.
\end{lem}

\begin{proof}
To see this let $\varepsilon =(16C)^{-1}, t_1=C\log|V_1|$,  $t_2=C\log|V_2|$ and assume that $t_1\leq t_2$. It is enough to show that our graph cannot have density larger than $1-\varepsilon$, as the other statement follows by looking at the (bipartite) complement $\ov{G}$. 

For the sake of contradiction, suppose the density $d(G) > 1-\varepsilon$, and let us count in two ways the number $M$ of stars $K_{t_1,1}$ which are formed by taking a single vertex in $V_2$ and $t_1$ of its neighbours.  
On one hand, each vertex $v\in V_2$ contributes $\binom{d(v)}{t_1}$ stars, giving us:
    \begin{align}
        \label{ineq: lower bound from convexity}
        M\geq \sum_{v\in V_2}\binom{d(v)}{t_1}\geq |V_2|\cdot \binom{e(G)/|V_2|}{t_1}\geq |V_2|\cdot \binom{(1-\varepsilon)|V_1|}{t_1} \geq |V_2| \cdot (1-2\varepsilon )^{t_1} \binom {|V_1|}{t_1}.
    \end{align}
Here we have used Jensen's inequality for the map $x\to \binom{x}{t_1}$ in the second inequality, and that $t_1 \leq \varepsilon |V_1|$ for the final step, as $|V_1| \geq n_C$. 

On the other hand, if $G$ is $K_{t_1,t_2}$-free then for each subset $S\subset V_1$ of $t_1$ vertices there are at most $t_2-1$ vertices in $V_2$ that can form a star with $S$, and so $M\leq (t_2-1)\binom{|V_1|}{t_1}$.
Combined with \eqref{ineq: lower bound from convexity} this gives $|V_2| \leq (t_2-1) (1-2\varepsilon )^{-t_1} \leq 
|V_2|^{1/2} e^{4\varepsilon t_1}$, using that $|V_2| \geq n_C$ and that $1 - x \geq e^{-2x}$ for $x \in [0, 1/2]$. It follows that $|V_2| \leq e^{8\varepsilon t_2} = |V_1|^{8C\varepsilon } =  |V_2|^{1/2}$, a contradiction.
\end{proof}

\begin{cor}\label{increasing-set-corr}
    For any $C>1$ there exists $n_C\in \mathbb{N}$ such that the following holds true. Any $C$-bipartite-Ramsey graph $G = (V_1,V_2,E)$ with $|V_1|,|V_2|\geq n_C$ contains at least $2|V_1|/3$ vertices in $V_1$ which all have degrees between $(32C)^{-1}|V_2|$ and $(1-(32C)^{-1})|V_2|$.
\end{cor}

\begin{proof}
Suppose for the sake of contradiction that the conclusion is not true. Let $\varepsilon:=1/32C$ and suppose at most $2|V_1|/3$ vertices $v\in V_1$ have degrees between $\varepsilon|V_2|$ and $(1-\varepsilon)|V_2|$. Then, by the pigeonhole principle, without loss of generality we can assume that there is a set $U_1 \subset V_1$ with $|U_1| \geq |V_1|/6$ such that all vertices in $U_1$ have degree less than $\varepsilon|V_2|$. But then the edge density of the induced bipartite graph $G[U_1,V_2]$ is less than $\varepsilon$. On the other hand, it is easy to see that if $G$ is $C$-bipartite-Ramsey, since $6|U_1| \geq |V_1| \geq n_C$, the graph $G[U_1,V_2]$ is $2C$-bipartite-Ramsey. This is a contradiction by Lemma \ref{density-lemma}. 
\end{proof}

In particular, we can repeatedly make use of Corollary \ref{increasing-set-corr} inside a bipartite graph to obtain the following result, which will be useful later. 

\begin{lem}\label{private-ngh-lemma}
Given $C>1$ and a natural number $L$ there is $n_0\in \mathbb{N}$ such that the following holds. Suppose that $G = (V_1, V_2, E)$ is a $C$-bipartite-Ramsey graph with $|V_i|\geq n_0$. Then, taking $\varepsilon_i:=(64C)^{-i}$ for all $i\in [L]$, one can find vertices $U:=\{u_1,u_2,\ldots,u_L\}\subset V_1$ such that $\big |N(u_i)\setminus\underset{j<i}{\bigcup}N(u_j) \big |\geq \varepsilon_i|V_2|$ and $\big |V_2\setminus\underset{j\leq i}{\bigcup} N(u_j) \big |\geq\varepsilon_i|V_2|$ for all $i\in[L]$.  
\end{lem}

\begin{proof}
By Corollary \ref{increasing-set-corr} we can pick $u_1\in V_1$ such that $d(u_1)\in [\varepsilon_1|V_2|,(1-\varepsilon_1)|V_2|]$. Our requirement is clearly satisfied for $i=1$. Now let us assume that we have found $u_1,u_2,\ldots,u_i$ with $i < L$ that satisfy our requirements and let us look for a vertex $u_{i+1}$.

Let $S_i:=V_1\setminus\{u_1,u_2,\ldots,u_i\}$ and $T_i:=V_2\setminus\underset{j\leq i}{\cup} N(u_j)$. Note that 
$|S_i| \geq |V_1| - L \geq |V_1|^{1/2}$ and that $|T_i| \geq (64C)^{-L}|V_2| \geq |V_2|^{1/2}$ since $|V_1|, |V_2| \geq n_0$. Therefore, the subgraph $G[S_i,T_i]$ is $(2C)$-bipartite-Ramsey. Thus, by Corollary \ref{increasing-set-corr}, it follows that we can find a vertex $u_{i+1}\in S_1$ such that $d^{T_i}_G(u_{i+1})\in [\varepsilon_1|T_i|,(1-\varepsilon_1)|T_i|]$. But this is precisely the vertex we were looking for, as then we have $|N(u_{i+1})\setminus \underset{j\leq i}{\cup}N(u_j)|\geq d^{T_i}_G(u_{i+1})\geq \varepsilon_1|T_i|\geq \varepsilon_{i+1}|V_2|$ and we also get that $|V_2\setminus\underset{j\leq i+1}{\cup} N(u_j)|\geq |T_i\setminus N^{T_i}_{G}(u_{i+1})|\geq \varepsilon_1|T_i|  \geq\varepsilon_{i+1}|V_2|$.

By repeating this step $L$ times we reach our conclusion. 
\end{proof}

\subsection{Richness and diversity}
\label{div-subsect}

We will use the notion of richness, which was introduced by Kwan and Sudakov \cite{kwan2018ramsey}. 

\begin{dfn}
Given $\delta,\varepsilon > 0$, a bipartite graph $G = (V_1, V_2 ,E)$ is $(\delta,\varepsilon)$\emph{-bipartite-rich} if for each $i\in \{1,2\}$ the following holds: for every set $W \subset V_i$ with $|W| \geq \delta |V_i|$ there are at most $|V_{3-i}|^{1/5}$ vertices $v \in V_{3-i}$ such that $|N(v) \cap W| \leq \varepsilon |V_{3-i}|$ or $|\overline{N}(v) \cap W| < \varepsilon |V_{3-i}|$.
\end{dfn}

By adapting a result of Kwan and Sudakov (Lemma 4 in \cite{kwan2018ramsey}) to the bipartite context, we show that the Ramsey setting guarantees richness. Perhaps the most striking difference here is that in a general Ramsey setting one needs to move to a subgraph to find richness, whereas in the bipartite setting Ramsey graphs already possess it. 

\begin{lem}\label{rich-lemma}
Given $C,\delta>0$ there is $\varepsilon>0$ and $n_0\in \mathbb{N}$ such that the following holds true. Every $C$-bipartite-Ramsey graph $G = (V_1, V_2 ,E)$ with $|V_1|, |V_2|\geq n_0$ is $(\delta,\varepsilon)$-bipartite-rich.
\end{lem}

\begin{proof}
It is enough to prove that the bipartite-richness condition holds when $i = 1$. So set $\varepsilon:=(200C)^{-1}$ and  suppose there is a set $U_1\subset V_1$ with $|U_1|\geq \delta|V_1|$ which contradicts the bipartite-richness condition -- more precisely that there is a set $W_2\subset V_2$ with $|W_2|\geq |V_2|^{1/5}$ such that $|N(v)\cap U_1|<\varepsilon|U_1|$ or $|\ov{N(v)}\cap U_1|<\varepsilon|U_1|$ for all $v\in W_2$. Without loss of generality, we can assume that there is a subset $U_2\subset W_2$ of size $|V_2|^{1/5}/2$ such that $|N(v)\cap U_1|<\varepsilon|U_1|$ for all $v\in W_2$. But this means that the edge density of $G[U_1,U_2]$ is less than $\varepsilon$.

On the other hand, as $|U_i| \geq |V_i|^{1/5}/2 \geq |V_i|^{1/6}$ as $|V_i| \geq n_C$, and since $G = G[V_1,V_2]$ is $C$-Ramsey, it follows $G[U_1, U_2]$ is $6C$-Ramsey. But by Lemma \ref{density-lemma} such a graph must have edge density at least $(196C)^{-1}>\varepsilon,$ which is a contradiction, thus proving our result.
\end{proof}

Next, we discuss the notion of \emph{diversity}, which was introduced by Bukh and Sudakov \cite{bukh} (in the non-bipartite setting). 

\begin{dfn} 
A bipartite graph $G = (V_1, V_2,E)$ is said to be $c$\emph{-bipartite-diverse} if for each $i\in \{1,2\}$ the following holds: every vertex $v\in V_i$ has $|\text{div}_G(v,w)|\leq c|V_{3-i}|$ for at most $|V_i|^{1/5}$ vertices $w \in V_i$.
\end{dfn}

We also introduce a useful diversity notion for pairs. We note that diversity for pairs was also considered by Kwan and Sudakov in \cite{kwan2018ramsey}, who considered multisets of neighbourhoods, but we require a different notion suitable for our later applications. 

\begin{dfn} 
A bipartite graph $G = (V_1, V_2,E)$ is said to be $(c,\alpha)$\emph{-pair-diverse} if for each $i\in \{1,2\}$ the following holds true for both $i = 1,2$. For each ordered pair $\bm{p}\in \binom {V_i}{2}$ with $|\text{div}(\bm{p})|\geq \alpha|V_{3-i}|$ there are at most $|V_i|^{1/5}$ pairwise vertex disjoint pairs $\bm{p}^\prime=(x,y) \in \binom {V_i}{2}$  such that $|\text{divb}_G(\bm{p})\setminus N(x)|\leq c|V_{3-i}|$ or $|\text{divb}_G(\bm{p})\setminus N(y)|\leq c|V_{3-i}|$.
\end{dfn}

\begin{lem}\label{rich-properties}
Let $G = (V_1, V_2 ,E)$ be a $(\delta,\varepsilon)$-bipartite-rich graph with $\delta\leq 1/2 $. Then:
\begin{enumerate}[(i)]
    \item $G$ is $\varepsilon/2$-diverse;
    \item $G$ is $(\alpha\varepsilon/2,\alpha)$-pair-diverse for all $\alpha\geq 2\delta$.
\end{enumerate}
\end{lem}

\begin{proof}
It is enough to show that each property holds when taking $i = 1$.

\emph{(i)} For each $v\in V_1$ either $|N(v)| \geq |V_2|/2$ or $|\ov{N(v)}|\geq |V_2|/2$.
 In the former case, all but at most $|V_1|^{1/5}$ vertices $w\in V_1$ we have $|N(v)\cap \ov{N(w)}|\geq\varepsilon |N(v)|\geq \varepsilon/2\cdot |V_2|$, and in the latter, for all but at most $|V_1|^{1/5}$ vertices $w\in V_1$ we have $|\ov{N(v)}\cap {N(w)}|\geq\varepsilon |\ov{N(v)}|\geq \varepsilon/2\cdot |V_2|$. In either case,
there are at most $|V_1|^{1/5}$ vertices $w\in V_1$ with $|\text{div}_G(v,w)| < \varepsilon/2\cdot |V_2|$, as desired.

\emph{(ii)} Suppose now there is an ordered pair $\bm{p}:=\{x,y\}\in \binom {V_1}{2}$ and a collection $Y$ of $|V_1|^{1/5}$ vertex disjoint pairs $\bm{p}_i:=(x_i,y_i) \in \binom {V_1}{2}$ such that $|\text{divb}(\bm{p})\setminus N(x_i)|\leq \alpha\varepsilon/2\cdot |V_2|$ for each $i$. Note that if $|\text{div}(\bm{p})|\geq \alpha|V_2|$ then $|\text{divb}(\bm{p})|\geq \alpha/2\cdot |V_2|$, as $\text{divb}(\bm{p})$ is simply the largest of the two sets $N(x)\setminus N(y)$ and $N(y)\setminus N(x)$, whose union is $\text{div}(\bm{p})$. But in this case there are $|V_1|^{1/5}$ distinct vertices $z_i \in \{x_i, y_i\}$ such that $|\text{divb}(\bm{p})\cap \ov{N(z_i)}|\leq\alpha\varepsilon/2\cdot |V_2|\leq \varepsilon|\text{divb}(\bm{p})|$, which contradicts the richness of the set $\text{divb}(\bm{p})$.  
\end{proof}

\subsection{Probabilistic tools}

Throughout the proof we will use Markov’s inequality, Chebyshev’s inequality, the Chernoff bound and Tur\'{a}n’s theorem. Statements and proofs of all of these can be found, for example, in \cite{alon04}. We will also need a probabilistic variation of the Erd\H{o}s-Littlewood-Offord theorem, a proof of which can be found, for instance in \cite{JKLY}, or more generally as a consequence of the Doeblin $-$ Kolmogorov $-$ Levy $-$ Rogozin theorem.

\begin{thm}\label{problitoff}
Fix some non-zero parameters $a_1,a_2,\dots,a_n\in \mathbb{R}$ and $\alpha\in(0,0.5]$, then let $p_1,p_2,\dots,p_n\in [\alpha,1-\alpha]$. Suppose that $X_1,X_2,\dots, X_n$ are independent Bernoulli random variables with $X_i\sim Be(p_i)$. Then the following holds: $$\displaystyle\max_{x\in \mathbb{R}} {\mathbb P} \bigg ( \displaystyle\sum_{i=1}^n a_iX_i=x \bigg ) =O_{\alpha }(n^{-1/2}).$$
\end{thm}

The following natural proposition will be useful in our proof of Theorem \ref{main-theorem}.

\begin{prop}\label{markov-chain-setresult}
Given an integer $d\geq 2$, there is an integer $n_0:=n_0(d)$ such that if $X$ is a uniformly chosen subset of $[n]$, where $n\geq n_0$, then for any integer $k$ with $0\leq k<d$ one has: $$\frac{1}{d+1}\leq \bP\big(|X|\equiv k\pmod{d} \big)\leq\frac{1}{d-1}.$$
\end{prop} 

Comparable statements already exist in the literature (see for example Lemma 2.3 in \cite{AHK} for a more quantitative behaviour), but to keep the paper contained we outline a simple proof, showing that it can be easily deduced from a standard result about stochastic processes. The next few paragraphs represent a brief introduction to this topic and, for more details, the reader can refer to \cite{norris1997}.    

A \emph{Markov chain} is a sequence $\bm{X}:=(X_n)_{n\geq 0}$ of random variables taking values in some common ground set $I$ such that, for all $n\geq 1$ and $i_0,i_1,\ldots,i_{n}\in I$, one has:
$$\bP\big(X_{n}=i_n\ |\ X_0=i_0;\ X_1=i_1;\ \ldots;\ X_{n-1}=i_{n-1}\big)=\bP\big(X_{n}=i_n\ |\ X_{n-1}=i_{n-1}\big).$$
\indent A Markov chain is called \emph{homogeneous} if, in addition, $p_{i,j}:=\bP(X_n=j|X_{n-1}=i)$ depends only on the states $i$ and $j$, not on the time $n$. These quantities are known as the transition probabilities of the chain. We are only interested in homogeneous chains and we can observe that in order to describe such a chain it is enough to have the initial distribution $\lambda$ of $X_0$, given by $\lambda_i:=\bP(X_0=i)$, and the transition matrix $P:=(p_{i,j})_{i,j\in I}$. Moreover, if we write $p_{i,j}^{(n)}$ for $\bP(X_{k+n}=j|X_k=i)$ then it is easy to see that $p_{i,j}^{(n)}=(P^n)_{i,j}$.

A Markov chain $\bm{X}$ on the set $I$ with transition matrix $P:=(p_{i,j})_{i,j\in I}$ is called \emph{irreducible} if for all $i,j\in I$ there is $n\geq 0$ such that $p_{i,j}^{(n)}>0$. The \emph{period} of a state $i\in I$ is defined to be the greatest common divisor of the set $\{n\geq 1:p_{i,i}^{(n)}>0\}$. The Markov chain $\bm{X}$ is called \emph{aperiodic} if all its states have period $1$. 

We say $\pi = (\pi _i)_{j\in I}$ is a \emph{stationary distribution} for a Markov chain $\bm{X}$ if starting the chain from $X_0$ with distribution $\pi$ implies that $X_n$ has distribution $\pi$ for all $n\geq 1$. As the distribution of $X_n$ is given by $\pi P^n$, where $P$ is the transition matrix, we deduce that $\pi$ is a stationary distribution if and only if $\pi P=\pi$.

\begin{thm}\label{stability-lemma}
Suppose $\bm{X}$ is an irreducible and aperiodic Markov chain on a ground set $I$ with transition matrix $P$, stationary distribution $\pi$ and any initial distribution. Then, for all $j\in I$, one has $\bP(X_n=j)\to \pi_j$ as $n\to \infty$.  
\end{thm}

\begin{proof}[Proof of Proposition \ref{markov-chain-setresult}]
We can view choosing $X$ as going through each number from $1$ to $n$ and independently tossing a fair coin for each to decide whether it is an element of $X$ or not. Thus $|X|\ (\text{mod }d)$ can be viewed as a Markov chain on $\{0,1,2,\ldots,d-1\}$ starting at $0$ and with transition probabilities $p_{i,i}=p_{i,i-1}=0.5$ for each $0\leq i<d$, where indices ar taken modulo $d$. This chain is aperiodic as $p_{i,i}^{(1)}=0.5$ and it is also irreducible as we can reach state $i$ from state $j$ in $i-j\ (\text{mod }d)$ steps with positive probability. It is easy to see that the stationary distribution $\pi$ is given by $\pi_i=1/d$. The conclusion now follows from Theorem \ref{stability-lemma}. 
\end{proof}

\subsection{Progressions in sumsets}

Given sets $A_1,\ldots, A_K \subset {\mathbb Z}$, the sumset of $A_1,\ldots, A_K$ is the set $A_1 + \ldots + A_K$ given by:
    \begin{align*}
        A_1 + \cdots + A_K = \big \{a_1 + \cdots + a_K: a_i\in A_i \mbox{ for all } i\in [K] \big \}.
    \end{align*}
\indent Much research has focused on the topic of estimating the size or understanding the structure of sumsets under certain assumptions on the sets (see for example \cite{long-arithmetic} or Chapter 2 of \cite{tao2006additive}).
\indent For our purposes, the following elementary estimate will suffice.

\begin{lem}\label{lem: sums of sets}
    Given $\delta , B>0$ there are $C, d_0 \geq 1$ such that the 
    following holds. Suppose that $A_1,\ldots, A_K 
    \subset [-M,M]$ with $K \geq CM$ and that 
    $|A_i| \geq \delta M \geq 2$ for all $i\in [K]$. 
    Then there is $a\in {\mathbb Z}$ and $d \in {\mathbb N}$ with 
    $1 \leq d \leq d_0$ such that: 
        \begin{align}
            \label{interval AP}
        \big \{ a + i d: 0 \leq i \leq BM^2  \big \}
             \subset 
        A_1 + A_2 + \cdots + A_K.
        \end{align}
\end{lem}

\begin{proof}
    We will prove the statement under the assumption that 
    $A_i \subset [0,M]$ for all $i\in [K]$. Note that the general case follows immediately by taking translations of the 
    sets, i.e. replacing $A_i$ with $A_i' = A_i-c_i$ for 
    $c_i = \min\{a: a\in A_i\}$. Indeed if \eqref{interval AP} holds for $A_1',\ldots, A_K'$ (which may now lie in $[0,2M]$ instead of $[-M,M]$) then \eqref{interval AP} also holds for $A_1,\ldots, A_K$ (possibly with a different value of $a$). 
    The same argument also allows us to assume that $0 \in A_i$ 
    for all $i \in [K]$.
    
    Next, we double count pairs $(m,j)\in [M]\times [K]$ for which $m$ is among the largest $\delta M/2$ elements in $A_j$. Counting by each set $A_j$, we deduce that there are $\delta KM/2$ such pairs. Therefore, there is $M^\prime\in [M]$ which appears in at least $\delta K/2$ of these pairs. Let $J:=\{j\in [K]:M^\prime\in A_j\}$, so that $|J|\geq \delta K/2$, and note that $|A_j\cap [0,M^\prime]|\geq \delta M/2$ as $M^\prime$ is one of the largest $\delta M/2$ elements in $A_j$. Recalling that also $0\in A_j$, by restricting all the sets to $[0,M']$, then eventually reordering them and slightly adjusting the parameter $C$, we can now assume that $A_i\subset [0,M]$ for all $i\in [K]$, but with $\{0,M\} \subset A_i$ as well.  
    
    
    Given a set $S \subset {\mathbb Z}$, we will let 
    $\overline{S} := \{a \in [0,M-1]: a \equiv s 
    \mod M \mbox{ for some } s \in S\}$. For each 
    $i\in [K]$ let $S_i := \overline{A_1 + \cdots + A_i}$. 
    By reordering the sets $A_1,\ldots, A_K$, we 
    may assume that $|S_{i+1}| > |S_{i}|$ for 
    $i \leq K'$ and that $S_i = S_{K'}$ for all 
    $i \geq K'$. Observe that $S_{K'}\subset [0,M-1]$ and that
    $|S_{K'}| \geq |S_1| + (K'-1) \geq K'+ 1$, therefore $K' < M$.
    
    Now, as $0 \in A_i$ for all $i\in [K]$, we have 
    $0 \in S_{K'} \subset \overline{S_{K'} + A_{K'+1}} = 
    S_{K'+1}$. As $|S_{K'+1}| = |S_{K'}|$, it follows 
    that the set $S_{K'}$ contains the subgroup of ${\mathbb Z}_M$ 
    generated by $A_{K'+1}$. However, recalling that $|S_{K'}| = |S_{K'+1}| \geq |A_{K'+1}| \geq \delta M$, we obtain that there is some $d \leq d_0(\delta )$ 
    with $d|M$
    such that $\{id : 0\leq i \leq M/d\} \subset S_{K'} \subset S_M$. 
    
    To proceed with the last step, recall that $\{0,M\} \in A_i$ 
    for all $i\in [K]$. Using this, it is easy to see 
    that as 
    $\{id : 0\leq i \leq M/d\} \subset S_{K'} 
    \subset S_M$, we have $\{M^2 + id: 0 \leq i \leq M/d\} 
    \subset A_1 + \cdots + A_{2M}$. More generally, as 
    $d | M$, we have: 
        $$
        \big \{M^2 + id + jM: 0 \leq i \leq M/d, \mbox{ } 0 \leq j \leq K-2M \big \}     \subset 
        A_1 + \cdots + A_{K},
        $$
    and so \eqref{interval AP} holds by taking $a = M^2$ and $K \geq CM 
    \geq (B + 2)M$. This completes the proof.
\end{proof}

\section{Proof of Theorem \ref{main-theorem}}

\subsection{Pair-stars and pair-matchings}

We start this section by defining two constructions which will be of central importance in our attempt to find induced subgraphs of many sizes. Let $G = (V_1, V_2 ,E)$ be a bipartite graph.

An $\varepsilon$\emph{-pair-star} of size $k$ rooted at $x_0$ associated to $V_1$ is a set ${\cal P}_S = \{x_0, x_1,x_2,\ldots,x_k\} \subset V_1$ which satisfies the following properties:
    \begin{enumerate}[(i)]
        \item $|d_G(x_j)-d_G(x_0)|\leq {|V_2|}^{0.5}$ for all $j\in [k]$;
        \item $|\text{div}(x_i,x_j)|\geq \varepsilon|V_2|$ for all $i\neq j$ in $[0,k]$.
    \end{enumerate}
We define the head of ${\cal P}_S$ to be the set $H({\cal P}_S) = \{x_0\}$.

An $\varepsilon $\emph{-pair-matching} of size $k$ associated to $V_1$ is a collection  ${\cal P}_M = \{\bm{p}_1,\bm{p}_2,\ldots,\bm{p}_k\}$ of vertex disjoint ordered pairs of vertices in $V_1$ which satisfy the following properties:
    \begin{enumerate}[(i)]
        \item  ${\bf p}_i = (x_i, y_i)$ for all $i\in [k]$ and the pairs ${\bf p}_1, \ldots, {\bf p}_k$ are vertex disjoint;
        \item $\dd{\bm{p}_i}{}\leq |V_2|^{0.5}$ for all $i\in[k]$;
        \item $|\text{divb}(\bm{p}_i)\setminus N(x_j)|\geq \varepsilon |V_2|$ and  $|\text{divb}(\bm{p}_i)\setminus N(y_j)|\geq \varepsilon |V_2|$ for all $i \neq j$ in $[k]$.
    \end{enumerate}
We define the head of ${\cal P}_M$ to be $H({\cal P}_M) = \{x_i: i \in [k]\}$.    
    
\indent By swapping $V_1$ and $V_2$ above, we can define \emph{pair-stars} and \emph{pair-matchings} associated to $V_2$. Our result in this subsection gives large pair-stars or large pair-matchings in $C$-Ramsey graphs.

\begin{lem}\label{lemma-matching-clique}
Given $C>1$ there is $\varepsilon >0$ and $n_C\in \mathbb{N}$ such that the following holds. Suppose that $G = (V_1, V_2 ,E)$ is a $C$-bipartite-Ramsey graph with $|V_1|\geq |V_2|/2 \geq n_C$. Then either:
\begin{itemize}
    \item $V_1$ contains an $\varepsilon $-pair-star of size $|V_1|^{0.5}$, or 
    \item $V_1$ contains an $\varepsilon $-pair-matching of size $|V_1|^{0.5}$.
\end{itemize}
\end{lem}

\begin{proof}
Let us start by dividing the interval $[0,|V_2|]$ into $\ell :={|V_2|}^{1/2}$ disjoint intervals $I_1,I_2,\ldots I_l$ of length ${|V_2|}^{1/2}$ and let, for each interval $I_j$, $n_j$ denote the number of vertices in $V_1$ whose degree lie in $I_j$. Now any vertices $x, y \in I_j$ with $d(x) \geq d(y)$ give an ordered pair $\bm{p} = (x,y)$ with $\dd{{\bf p}}{} \leq {|V_2|}^{1/2}$. Obviously $n_1+n_2+\ldots + n_{\ell}=|V_1|$, so by Jensen's inequality the collection $\mathcal{S}$ of such ordered pairs $\bm{p} = (x,y) \in \binom{V_1}{2}$ has size:
    \begin{align*}
        |{\cal S}| 
            \geq 
        \sum_{j=1}^{\ell}\binom{n_j}{2}\geq \ell \cdot \binom{|V_1|/\ell}{2}
            \geq 
        \frac {|V_1||\big(|V_1|/\ell-1\big)}{2} \geq \frac{|V_1|^{1.5}}{4},
    \end{align*}
using $\ell = |V_2|^{0.5} \leq (2|V_1| )^{0.5}$ and $|V_1| \geq n_C$. 
By Lemmas \ref{rich-properties} and \ref{rich-lemma} we deduce that our graph $G$ is $\varepsilon_0$-diverse, where $\varepsilon_0:=(400C)^{-1}$. Therefore, at most $|V_1|^{1.2}=o(|\mathcal{S}|)$ pairs from $\mathcal{S}$ fail to be ${\varepsilon}_0$-diverse. By removing such pairs we obtain a collection $\mathcal{S}_0$ of size at least $|V_1|^{1.5}/8$ such that $\dd{\bf p}{} \leq |V_2|^{1/2}$ and $|\text{div}(\bm{p})|\geq \varepsilon_0|V_2|$ for all $\bm{p}\in \mathcal{S}_0$.

Next, we view the elements of $\mathcal{S}_0$ as the (unordered) edges of a graph $H$ on $V(G)$. We claim that either $H$ has a matching of size $m:= |V_1|^{0.75}/4$ or a set of $m+1$ edges that have a common vertex. To see why this is true, suppose that neither of these two events hold and let $M$ be a largest matching in $H$. Then $|M|<m$ and so all the edges in $e(H)\setminus M$ must be adjacent to an edge of $M$. But because of the other condition, each edge of $M$ has at most $2m$ other edges adjacent to it. Thus $e(H)<2m^2 \leq |\mathcal{S}_0|$, which is a contradiction.

Let us now observe that this set of edges will create our \emph{pair-star} or \emph{pair-matching}. Indeed, assume first that we have obtained $m+1$ edges $x_0x_1,x_0x_2,\ldots,x_0x_{m+1}$ in $\mathcal{S}_0$. Again, by diversity, each vertex $x_j$ has at most $|V_1|^{0.2}$ other vertices $x_i$ such that $|\text{div}(x_j,x_i)|<\varepsilon_0|V_2|$. Thus, by Tur\'{a}n theorem there is a set $I\subset[0, m+1]$, with $0 \in I$, of size at least $(m+1)/(1+|V_1|^{0.2}) \geq |V_1|^{0.5}$ such that $|\text{div}(x_j,x_i)|\geq \varepsilon_0|V_2|$ for all $i\neq j$ in $I$. Taking $x_0$ as the root gives our pair-star.

In the other case, assume that we have obtained $m$ disjoint ordered pairs ${\bf p}_1,\ldots, {\bf p}_m$ in ${\cal S}_0$, with ${\bf p}_i = (x_i, y_i)$. Then, by Lemmas \ref{rich-lemma} and \ref{rich-properties} our graph $G$ is $((\varepsilon _0/2)^2, \varepsilon _0/2)$-pair-diverse, so for each pair $(x_j,y_j)$ there are at most $|V_1|^{0.2}$ other pairwise disjoint edges $x_iy_i$ such that $|\text{divb}(x_j,y_j)\setminus N(x_i)|< (\varepsilon _0/2)^2 |V_2|$ or $|\text{divb}(x_j,y_j)\setminus N(y_i)|< (\varepsilon _0/2)^2 |V_2|$. Setting $\varepsilon = (\varepsilon _0/2)^2$, by Tur\'{a}n, there is a set $I\subset [m]$ of size at least $m\left(1+|V_1|^{0.2}\right)^{-1} \geq |V_1|^{0.5}$ such that $|\text{divb}({\bf p}_j)\setminus N(x_i)|\geq \varepsilon |V_2|$ and $|\text{divb}({\bf p}_j)\setminus N(y_i)|\geq \varepsilon |V_2|$ for all $i\neq j$ in $I$. This represents our pair-matching. 
\end{proof}

%
%
%
%
%
%
%
%
%
%
%
%
%
%

\subsection{Degree control from pair-stars and pair-matchings}

Our next aim is to show that when picking a subset $W \subset V_2$ uniformly at random, there is a very good chance that \emph{pair-stars} and \emph{pair-matchings} associated to $V_1$ produce well distributed degree sets in $W$. The following lemma will be useful in this context.

\begin{lem}\label{small-prob-same-degree}
Let $G = (V_1, V_2, E)$ be a bipartite graph and suppose the subset 
$W \subset V_2$ is selected uniformly at random. Then:
\begin{enumerate}[(i)]
    \item if $x, y \in V_1$ with 
    $|\text{\emph{div}}(x,y)|\geq \delta|V_2|$ then 
    $\bP\left(d^{W}(x)=d^{W}(y)\right)
    =O_{\delta }\left(|V_2|^{-0.5}\right).$
    \item if $\bm{p}_1=(x_1,y_1)\in \binom {V_1}{2}$ and $\bm{p}_2=(x_2,y_2)\in \binom {V_1}{2}$ with $|\text{\emph{divb}}(\bm{p}_1)\setminus N(x_2)|\geq \delta|V_2|$ and $|\text{\emph{divb}}(\bm{p}_1)\setminus N(y_2)|\geq\delta|V_2|$ then 
    $ \bP\big( \dd  {{\bf p}_1} {W}=\dd {{\bf p}_2} {W} \big)=O_\delta\left(|V_2|^{-0.5}\right).$
\end{enumerate}
\end{lem}

\begin{proof}
(i) We will use a classical randomness exposure argument. Suppose we reveal the random set $W$ on $V_2\setminus \text{div}(x,y)$. Then, given such a choice, the difference $d^{W}(x)-d^{W}(y)$ becomes $d^{U}(x)-d^{U}(y)+\text{constant}$, where $U :=W \cap \text{div}(x,y)$. Now, $d^{U}(x)-d^{U}(y):=\sum_{v\in \text{div}(x,y)}\theta_vX_v$, where for each $v\in \text{div}(x,y)$ we have $\theta_v\in\{-1,1\}$ and $X_v\sim \text{Bern}(0.5)$. Since $|\text{div}(x,y)|\geq\delta|V_2|$, by Theorem \ref{problitoff}, the random variable $d^{W'}(x)-d^{W'}(y)$ hits any particular value with probability $O_\delta\left(|V_2|^{-1/2}\right)$. The conclusion follows from the law of total probability.   

(ii) The argument here is similar, but a little more involved. By the choice of the ordering we have $\mbox{divb}({\bf p}_1) = N(x_1) \setminus N(y_1)$. Note that the condition in the lemma now gives that the set $T = \mbox{divb}({\bf p}_1) \setminus N(x_2) = N(x_1) \setminus \big ( N(x_2) \cup N(y_1) \big )$ satisfies $|T| \geq \delta |V_2|$. Letting $X$ denote the random variable $X = \dd{{\bf p}_1}{W} - \dd{{\bf p}_2}{W}$ we have: 
    \begin{align*}
        X = d^{W}(x_1)-d^{W}(y_1) -d^{W}(x_2) + d^{W}(y_2).
    \end{align*}
Now note that if we expose the random set $W \cap (V_2 \setminus T)$ then the random variable $X$ reduces to $d^{W \cap T}(x_1)+d^{W \cap T}(y_2)-C$, where $C$ is some constant depending only on $W\cap (V_2 \cap T)$ -- crucially, here we use that $y_1$ and $x_2$ have no neighbour in $T$. But $d^{W \cap T}(x_1)+d^{W \cap T}(y_2)=\sum_{v\in T}\theta_vX_v$ where $\theta_v\in\{1,2\}$ is a constant for each $v\in T$ and $X_v\sim \text{Bern}(0.5)$ so that $\{X_v\}_{v\in T}$ are independent. By Theorem \ref{problitoff}, since $|T|\geq \delta|V_2|$, the random variable $d^{W \cap T}(x_1)+d^{W\cap T}(y_2)$ takes any particular value with probability at most $O_\delta\left(|V_2|^{-0.5}\right)$. The conclusion again follows from the law of total probability. 
\end{proof}

Let $G = (V_1, V_2, E)$ be a bipartite graph and select a vertex set $W \subset V_2$. Given an $\varepsilon $-pair star ${\cal P} = \{x_0,x_1,\ldots, x_k\}$ rooted at $x_0$ and associated to $V_1$, we write: 
    \begin{align*}
        A_{\cal P}^W
            :=
         \big \{d^W(x_i) - d^W(x_0): i\in [k] \big \} \cap \big [-3|V_2|^{0.5},3|V_2|^{0.5} \big ].
    \end{align*}
\indent The following lemma provides a useful estimate on the size of $A_{\cal P}^W$.

\begin{lem}\label{lemma-good-star}
Given $\varepsilon > 0$ there is $\delta >0$ such that the following holds. Let $G = (V_1, V_2, E)$ be a bipartite graph with $|V_1|\geq |V_2|/2$ and let ${\cal P} = \{x_0,\ldots, x_{|V_2|^{0.5}}\}$ be a $\varepsilon $-pair-star of size ${|V_2|}^{0.5}$ rooted at $x_0$ and associated to $V_1$. Suppose that a set $W \subset B$ is selected uniformly at random, then ${\mathbb P}\big (|A_{\cal P}^W| \geq \delta |V_2|^{0.5} \big ) \geq 3/4$.
\end{lem}

\begin{proof}
To begin, pick $j\in [|V_2|^{0.5}]$ and set $D_j := d^W(x_j) - d^W(x_0)$. By Chebyshev's inequality, we have $\left|D_j-\bE[D_j]\right|\leq 2{|V_2|}^{0.5}$ with probability at least $15/16$. As $\left|\bE[D_j]\right|\leq |V_2|^{1/2}/2$, it follows by triangle inequality that $\bP\left(|d^{W}(x_0)-d^{W}(x_j)|\leq 3|V_2|^{0.5}\right)\geq 15/16$.

Next, call a vertex $x_j$ with $j\in [|V_2|^{0.5}]$ \emph{bad} if $|d^{W}(x_0)-d^{W}(x_j)|\leq 3|V_2|^{0.5}$ and \emph{good} otherwise. Note $\bP(x_j\text{ is bad})\leq 1/16$ and so $\bE\left[\big|\{j\in [|V_1|^{0.5}]:x_j\text{ is bad}\}\big|\right]\leq |V_2|^{0.5}/16$. If $U$ denotes the set of \emph{good} vertices, then it follows by Markov that $\bP\left(|U|\geq |V_2|^{0.5}/2\right)\geq 7/8$.

Now we build a graph $H$ on $\{x_1,\ldots, x_{|V_2|^{1/2}}\}$ where we join two vertices by an edge in $H$ if their degrees in $W$ are equal. By Lemma \ref{small-prob-same-degree} (i) there is an edge in $H$ between two vertices $x_i$ and $x_j$ with probability $O_{\varepsilon } \left(|V_2|^{-1/2}\right)$. Therefore $\bE[e(H)]=O_{\varepsilon }\left((|V_2|^{0.5})^2\cdot |V_2|^{-0.5}\right)$ and by Markov we easily deduce that, with probability at least $7/8$, one has $e(H)=O_{\varepsilon }\left(|V_2|^{0.5}\right)$.

Combining our estimates, by applying the union bound we find that $|U| \geq |V_2|^{0.5}/2$ and that $e(H[U]) \leq e(H)=O_{\varepsilon}\left(|V_2|^{0.5}\right)$ with probability at least $3/4$. But then the average degree of $H$ is of order $O_{\varepsilon }\left(1\right)$, hence by Tur\'{a}n $H$ contains an independent set of size $\Omega_{\varepsilon}\left({|V_2|^{0.5}}\right)$. This set though gives us precisely the vertices with pairwise distinct degrees that we sought. 
\end{proof}

Let $G = (V_1, V_2, E)$ be a bipartite graph, and set $W \subset V_2$. Given an $\varepsilon $-pair-matching ${\cal P} = \{{\bf p}_1,\ldots, {\bf p}_k\}$ associated to $V_1$, we write: 
    \begin{align*}
        A_{\cal P}^W
            := 
        \big \{\dd{{\bf p}_i}{W}: i\in [k] \big \} \cap \big [-3|V_2|^{0.5},3|V_2|^{0.5} \big ].
    \end{align*}
\indent The following lemma gives the analogous estimate for $A_{\cal P}^W$ when ${\cal P}$ is an pair-matching, and can be proven in exactly the same way as in Lemma \ref{lemma-good-star}, replacing Lemma \ref{small-prob-same-degree} (i) in the proof with Lemma \ref{small-prob-same-degree} (ii) instead.

\begin{lem}\label{lemma-good-matching}
Given $\varepsilon > 0$ there is $\delta >0$ such that the following holds. Let $G = (V_1, V_2, E)$ be a bipartite graph with $|V_1|\geq |V_2|/2$ and let ${\cal P}$ be a $\varepsilon $-pair-matching in $V_1$ of size $|V_2|^{0.5}$. Suppose that a set $W \subset B$ is selected uniformly at random, then ${\mathbb P}\big (|A_{\cal P}^W| \geq \delta |V_2|^{0.5} \big ) \geq 3/4$.
\end{lem}

\subsection{Breaking modular obstructions}

\begin{lem}\label{modular-lem}
Given a natural number $d > 1$ and $\varepsilon > 0$ there are $L, D >0$ such that the following holds. Suppose that $G = (V_1, V_2, E)$ is a bipartite graph and that $\{u_1,\ldots, u_L\} \subset V_1$ such that $S_i = N_G(u_i) \setminus \cup _{j < i} N_G(u_j)$ satisfies $|S_i| \geq D$ for all $i\in [L]$. Then, if a subset $W$ is chosen uniformly at random from $V_2$, with probability at least $1 - \varepsilon $ for every pair $(k,m)$ with $0 \leq k \leq m$ and $2 \leq m \leq d$ there is $i\in [L]$ such that $d_G^{W}(u_i) \equiv k \mod m$.
\end{lem}

\begin{proof}
Suppose that $W \subset V_2$ is selected as in the lemma. To begin, we fix a pair $(k,m)$ so that $0 \leq k \leq m$ and $2 \leq m\leq d$. For $i \in [L]$ let $E_i$ denote the event $E_i := \{d^W_G(u_i) \not \equiv k \mod m\}$. Our first aim is to upper bound ${\mathbb P}\big ( \cap _{i\in [L]} E_i )$. To do this, we can observe that:
\begin{align*}
    \bP\big(\cap_{i=1}^L E_i\big)=\prod_{i=1}^L \bP\left(E_i|\cap_{j<i}E_j\right). 
\end{align*}
\indent Now the event $\cap _{j<i}E_i$ is entirely determined by the choice of $W' = W \cap (V_2 \setminus S_i)$. Thus, to upper bound $\bP\left(E_i|\cap_{j<i}E_j\right)$, it suffices to upper bound $\bP\left(E_i|W' = W_0\right )$ for each choice of $W_0$. However, given such a choice of $W_0$, the conditional probability $\bP\left(E_i|W'=W_0\right)$ becomes $\bP\big(d^{W\cap S_i}_G(u_i)\not\equiv k_i\ (\text{mod }m)\big)$ for some $0\leq k_i<m$. Besides this, from our hypothesis we know that $W \cap S_i \sim \mbox{Bin}(|S_i|,0.5)$ and $|S_i| \geq D$, so by Proposition \ref{markov-chain-setresult} this gives $\bP\left(E_i|\cap_{j<i}E_j\right) \leq \max _{W_0} {\mathbb P}(E_i| W' = W_0) \leq d(d+1)^{-1}$. Combining all these, we obtain:
    \begin{align*}
        \bP\big(\cap_{i=1}^L E_i\big) 
            \leq 
        \prod_{i=1}^L \bP\left(E_i|\cap_{j<i}E_j\right)
            \leq 
        \Big ( 1 - \frac {1}{d+1} \Big ) ^{L}\leq \exp\left(-L(d+1)^{-1}\right).
    \end{align*}
\indent To finish the proof, let $F_{(k,m)}$ denote the event that $\{d_G^W(u_i) \not \equiv k \mod m \mbox{ for all } i\in [L]\}$. We have shown that ${\mathbb P}(F_{(k,m)}) \leq \exp ( -L(d+1)^{-1} )$. It follows that the probability that some congruence is not obtained is ${\mathbb P} \big ( \cup _{(k,m)} F_{k,m} \big ) \leq \sum _{(k,m)} {\mathbb P}(F_{k,m}) \leq d^2 \exp (-L(d+1)^{-1} ) \leq \varepsilon $, provided that $L$ (and $D$ from above) are sufficiently large.
\end{proof}

\subsection{Completing the proof of Theorem \ref{main-theorem}}

The following lemma is the final key step in our proof.

\begin{lem}\label{lem-interval-of-sizes}
Given $C>1$, there are constants $a_C > 0$ and $n_C\in \mathbb{N}$ such that the following holds. Every $C$-bipartite-Ramsey graph  
$G = (V_1, V_2, E)$ with $|V_1|,|V_2|\geq n_C$ has: 
$$\big [a_C|V_1||V_2|, 2a_C|V_1||V_2| \big ] \subset \big \{e(G[U]): U \subset V(G) \big \}.$$
\end{lem}

\begin{proof}
We start by fixing parameters $1\geq  C^{-1} \gg \varepsilon \gg \delta \gg d_0^{-1} \gg L^{-1} \gg C_0^{-1} \gg c \gg n_C^{-1} > 0$. Without loss of generality, we may assume that $|V_1| \geq |V_2|$. We also fix an arbitrary set $V_2' \subset V_2$ with $|V_2'| = c|V_2|$ and a partition $V_1 = U_1 \cup U_2 \cup U_3$, where $|U_i| \geq |V_1|/4$ for $i\in [3]$. 

To begin, we claim that there are ${\cal P}_1,\ldots, {\cal P}_K$ with {$K := C_0|V_2|^{0.5}$} such that each ${\cal P}_i$ is either an $\varepsilon$-pair-star or a $\varepsilon $-pair-matching in $G[U_1, V_2']$ of size $|V_2'|^{0.5}$ and so that the ${\cal P}_i$'s are all vertex disjoint. To see this, suppose that we have found ${\cal P}_1,\ldots, {\cal P}_i$ and now seek ${\cal P}_{i+1}$. Let $U_1'$ denote the subset of $U_1$ with the vertices of $\cup _{j\leq i} {\cal P}_j$ removed and apply Lemma \ref{lemma-matching-clique} to $G[U_1', V_2']$, noting that as $|U_1'| \geq |U_1| - 2 K |V_2'|^{0.5} \geq |U_1| - 2(C_0|V_2|^{0.5}) (c|V_1|)^{0.5} \geq |U_1|/2 \geq |V_1|/8$, making $G[U_1', V_2']$ a $2C$-Ramsey graph. By Lemma \ref{lemma-matching-clique}, we can find a $\varepsilon $-star-pair or a $\varepsilon $-pair-matching ${\cal P}_{i+1}$ of size $|U_1'|^{0.5} \geq |V_2'|^{0.5}$, which gives the desired set (perhaps after removing some of its elements). Thus such a collection ${\cal P}_1,\ldots, {\cal P}_K$ exists.

Our next step is to apply Lemma \ref{private-ngh-lemma} to $G[U_2, V_2']$, which is again $2C$-Ramsey, to find a set $S_2 = \{s_1,\ldots, s_L\} \subset U_2$ such that $|N(s_i) \setminus ( \cup _{j< i}N(s_j) )| \geq |V_2|^{0.5}$ for all $i\in [L]$. It is possible to do this since $C^{-1} \gg L^{-1} \gg n_C^{-1}$.

The final step of preparation is to select a set $W \subset V_2'$ uniformly at random. For each $i\in [K]$ let $A_i := A_{{\cal P}_i}^W$. Now consider the following events: \begin{itemize}
    \item Let ${\cal E}_1$ denote the event that, for at least $K/2$ values $i\in [K]$, we have $|A_i| \geq \delta |V_2'|^{1/2}$. By Lemmas \ref{lemma-good-star} and  \ref{lemma-good-matching}, and using that $\varepsilon \gg \delta $, we see that for any $A_i$ we have $|A_i| \geq \delta |W|^{1/2}$ with probability at least $3/4$, so a simple application of Markov gives ${\mathbb P}({\cal E}_1) \geq 1/2$; 
    
    \item Secondly, let ${\cal E}_2$ denote the event that, for every $0 \leq k \leq m$ and $2 \leq m \leq d_0$, there exists $s \in S_2$ with $d^W_G(s) \equiv k \mod m$. By the choice of $S_2$, Lemma \ref{modular-lem} gives 
    ${\mathbb P}({\cal E}_2) \geq 7/8$; 
    
    \item Lastly, let ${\cal E}_3$ denote the event that $|W| \geq |V_2'|/4$. By Chebyshev's inequality ${\mathbb P}({\cal E}_3) \geq 7/8$. 
\end{itemize}
 It follows from the union bound that ${\mathbb P}({\cal E}_1 \cap {\cal E}_2 \cap {\cal E}_3) \geq 1/2$. Fix a choice of $W \subset V_2' \subset V_2$ for which all three events occur. By reordering, we get that $|A_i| \geq \delta |V_2'|^{1/2} \geq \delta |W|^{1/2}$ for $i\in [K/2]$. 

We are now ready to complete the proof. Take $V_1^{\mbox{init}}$ to be the union of the heads of ${\cal P}_i$, i.e. $V_1^{\mbox{init}} := \cup _{i \leq L/2} H({\cal P}_i)$, and $e_0 := e(G[V_1^{\mbox{init}}, W])$. Take $a \in A_i$ and note the following:  
\begin{enumerate}[(i)]
    \item If ${\cal P}_i = \{x_0,\ldots, x_{|V_2'|^{0.5}}\}$ is a pair-star rooted at $x_0$ then by definition 
    $d^W(x_j) - d^W(x_0) = a$ for some $j\in [|V_2'|^{1/2}]$. Removing $x_0$ from $V_1^{\mbox{init}}$ and adding $x_j$ changes the number of edges in the resulting graph by exactly $a$. 
    \item Similarly, if ${\cal P}_i$ is a pair-matching then there is 
    ${\bf p} = (x,y) \in {\cal P}_i$ such that $\dd{{\bf p}}{W} = a$. 
    Removing $x$ from $V_1^{\mbox{init}}$ and adding $y$ to it changes 
    the number of edges by exactly $a$.
    \item The edits from different ${\cal P}_i$'s can be performed independently of one 
    another, resulting in the same changes to the edge size given in (i) and (ii).
\end{enumerate}
From observations (i)$-$(iii) we can immediately deduce that: 
$$\{e_0\} + A_1 + A_2 + \cdots + A_{K/2} \subset \big \{ e(G[U,W]): U \subset U_1 \big \}.$$ 
\indent By definition of $A_i:=A_{{\cal P}_i}^W$, we have 
$A_i \subset [-3|V_2'|^{1/2}, 3|V_2'|^{1/2}] \subset [-6|W|^{1/2}, 6|W|^{1/2}]$ and $|A_i| \geq \delta |W|^{1/2}$. By taking $M = 6|W|^{1/2}$ and $B = 1$ in Lemma \ref{lem: sums of sets}, as $\delta ,1 \gg C_0^{-1}, d_0^{-1}$, it follows that there is $a \in {\mathbb Z}$ and $1 \leq d \leq d_0$ such that \eqref{interval AP} holds, which gives: 
\begin{align}
    \label{equation: edge interval}
    \{e_0 + a + id: 0 \leq i \leq 3|W|\} \subset
    \{e_0\} + \{a + id: 0 \leq i \leq M^2\} 
    & \subset 
    \{e_0\} + A_1 + A_2 + \cdots + A_{L/2} \nonumber \\
    & \subset
    \{e(G[U, W]): U \subset U_1 \}.
\end{align}
\indent Furthermore, as $d \leq d_0$, by the choice of $S_2$ there are $s_{i_0}, s_{i_1},\ldots, s_{i_{d-1}} \in S_2$ with $d^W(s_j) \equiv j \mod d$ for each $j\in [0,d-1]$. Combined with \eqref{equation: edge interval}, setting $e_1 = e_0 + a + |W|$, this gives: 
\begin{align}
    \label{equation: short interval}
    \big [e_1, e_1 + |W| \big ]
        \subset 
    \big \{e(G[U \cup \{s\}, W]): U \subset U_1, s \in S_2 \big \}.
\end{align}
\indent To finish the proof, note that $|U_3| \geq |V_1| /4 \geq |V_1|^{1/2}$ and $|W| \geq |V_2'|/4 \geq |V_2|^{1/2}$. As $G$ is $C$-Ramsey, we get that $G[U_3, W]$ is $(2C)$-Ramsey. It follows from Corollary \ref{increasing-set-corr} that at least $2|U_3|/3$ vertices in $U_3$ have degrees between $(64C)^{-1}|W|$ and $(1-(64C)^{-1})|W|$ in $G[U_3, W]$. In particular, $\{e(G[U',W]) : U' \subset U_3\}$ contains an element from each interval $[|W|i, |W|(i+1)]$ with $i\in [0,(96C)^{-1}|U_3|]$. However, 
$e(G[U \cup \{s\} \cup U', W]) = e(G[U \cup \{s\}, W]) + e(G[ U', W])$ for any $U \subset U_1, s \in S_2$ and $U' \subset U_3$. Therefore, we deduce from \eqref{equation: short interval} that:
\begin{align*}
    \big [ e_1 + |W|, e_1 + |W| + (96C)^{-1}|U_3||W| \big ]
        \subset 
    \big \{ e(G[U \cup \{s\} \cup U', W]): U \subset U_1, s \in S_2, 
    U' \subset U_3 \big \}.
\end{align*}
\indent We can finally claim that the lemma holds with $a_C := c/4000C$. To see this, note that: 
\begin{align*}
    e_1 + |W| = e_0 + a + |W| 
        & \leq 
        (K/2)|V_2'|^{0.5}|W| + (K/2)(6|W|)^{0.5}|W| + |W|\\ 
        &\leq 4K|W|^{1.5} \leq 4(C_0|V_2|)^{0.5}(c|V_2|)^{1.5} = 4c^{1.5}(C_0)^{0.5}|V_2|^2 \leq a_C|V_1||V_2|,
\end{align*}
where we have used that $|W| \leq c|V_2|$, that $c \ll C_0^{-1}, C^{-1}$, and that $|V_2| \leq |V_1|$.\\
\indent Since $a_C|V_1||V_2| = (c/4000C)|V_1||V_2| \leq (200C)^{-1}|U_3||W|$, it is quite easy to observe that $[a_C|V_1||V_2|, 2a_C|V_1||V_2|] \subset \{e(G[U]): U \subset V(G)\}$, as required. This completes the proof.
\end{proof}

With Lemma \ref{lem-interval-of-sizes} in hand, it is now easy to complete the proof of Theorem \ref{main-theorem}.

\begin{proof}[Proof of Theorem \ref{main-theorem}]
By making $\alpha $ small enough we may assume, without losing generality, that $|V_1|\geq |V_2|$ and that $|V_2|$ is large enough so that our estimates below hold. We can also assume that $C>1$ since the Ramsey condition still holds when we increase $C$. 

First note that as $G = (V_1, V_2, E)$ is $C$-Ramsey, there is a vertex $v \in V_2$ with degree at least $(32C)^{-1}|V_1|$ in $V_1$, by Corollary \ref{increasing-set-corr}. But then the induced subgraphs of the form $G[W, \{v\}]$,
where $W \subset V_1$, give all edge sizes in $[0,(32C)^{-1}|V_1|]$ (and so in particular the analogue of the Alon $-$ Krivelevich $-$ Sudakov theorem from \cite{AKS-size} is easier in the bipartite Ramsey context).

On the other hand, observe that by fixing any $W_1 \subset V_1$ and $W_2 \subset V_2$ with $|W_i| \geq |V_i|^{1/3}$ for $i \in \{1,2\}$, the graph $H = G[W_1, W_2]$ is $(3C)$-Ramsey, as $G$ is $C$-Ramsey. It follows by applying Lemma \ref{lem-interval-of-sizes} to $H$ that there are induced subgraphs of each size in $[a_{3C}|W_1||W_2|, 2a_{3C}|W_1||W_2|]$. It follows, taking $M = \{(m_1, m_2): |V_i|^{1/3} \leq m_i \leq  |V_i| \mbox{ for } i \in \{1,2\}\}$, that 
$G$ contains a subgraph of each size in $\cup _{(m_1,m_2) \in M} [a_{3C}m_1m_2, 2a_{3C}m_1m_2]$. It can be easily seen that these sets together cover the interval 
$[a_C(|V_1||V_2|)^{1/3}, 2a_C |V_1||V_2|]$. 

Finally, note that the intervals $[0,(32C)^{-1}|V_1|]$ and $[a_{3C}(|V_1||V_2|)^{1/3}, 2a_{3C} |V_1||V_2|]$ together cover 
$[0, 2a_{3C} |V_1||V_2|]$, since $|V_1| \geq |V_2|$ and $|V_2|$ is large enough, thus completing the proof of the theorem.
\end{proof}

\section{Concluding remarks}

In this paper we have proven a bipartite analogue of the Erd\H{o}s--McKay conjecture, showing that any $C$-bipartite-Ramsey graph $G = (V_1, V_2, E)$ must contain induced subgraphs of all sizes in $[0, \Omega _C(|V_1||V_2|)]$. Of course the Erd\H{o}s--McKay conjecture itself remains an outstanding problem, and we hope that some of our ideas will be useful in future approaches here. 

Another interesting direction is to understand the effect of weakening the Ramsey hypothesis in Theorem \ref{main-theorem}, as considered already in a number of other settings (see e.g. \cite{alon1989graphs}, \cite{narayanan2018induced}, \cite{LP2022distinct}). That is, what are we able to say about the edge sizes of induced subgraphs of a bipartite graph $G = (V_1, V_2, E)$, which do not contain induced copies of $K_{t_1,t_2}$ or $\overline{K_{t_1, t_2}}$, where $t_1$ and $t_2$ are now general parameters?  

Lastly, we note that Narayanan, Sahasrabudhe and Tomon in \cite{narayanan2017multiplication} showed that any bipartite graph with $m$ edges must have induced subgraphs of $\Omega (m / \log ^{12}(m))$ different sizes. An obvious upper bound here for the number of such sizes is $m$, though if $m$ is a perfect square $m = k^2$ then the complete bipartite graphs $K_{k,k}$ allows only $m / (\log m)^{0.086... + o(1)}$ edge sizes. The authors of \cite{narayanan2017multiplication} conjecture that when $m = k^2$ the graph $K_{k,k}$ is extremal here. It would be interesting to understand whether our approach, perhaps combined with stability arguments, could be applied to improve knowledge here.

\printbibliography

\end{document}